\documentclass[aop]{imsart2}
\usepackage{tikz}
\usepackage{neuralnetwork}
\usepackage{amsthm,amsmath,amsfonts,amssymb}
\usepackage{cite}
\usepackage{appendix}
\usepackage{cancel}
\usepackage{mathptmx}
\usepackage[T1]{fontenc} 
\usepackage[utf8]{inputenc}
\usepackage{lmodern}

\theoremstyle{plain}

\newtheorem{theorem}{Theorem}[section]

\newtheorem{lemma}[theorem]{Lemma}
\theoremstyle{remark}

\def \cL {\mathcal L}

\def \i {\iota}

\def \E {\mathbb{E}}

\def \N {\mathbb{N}}

\def \P {\mathbb{P}}

\def \R {\mathbb{R}}

\def \Z {\mathbb{Z}}

\def \cG {\mathcal{G}}

\def \cL {\mathcal{L}}

\def \lp {\left(}
\def \rp {\right)}
\def \ls {\left\{}
\def \rs {\right\}}

\DeclareMathOperator{\TW}{TW}
\DeclareMathOperator{\GUE}{GUE}

{}

\usepackage{hyperref}
\begin{document}
\begin{frontmatter}

\title{A universality property for large deviations of RWRE close to the axis}
%\title{A sample article title with some additional note\thanksref{t1}}
\runtitle{A universality property  for RWRE}
%\thankstext{T1}{A sample additional note to the title.}

\begin{aug}
%%%%%%%%%%%%%%%%%%%%%%%%%%%%%%%%%%%%%%%%%%%%%%%
%% Only one address is permitted per author. %%
%% Only division, organization and e-mail is %%
%% included in the address.                  %%
%% Additional information can be included in %%
%% the Acknowledgments section if necessary. %%
%%%%%%%%%%%%%%%%%%%%%%%%%%%%%%%%%%%%%%%%%%%%%%%
\author[A,B]{\fnms{Pablo} \snm{Groisman}\ead[label=e1, mark]{pgroisma@dm.uba.ar}},
\author[B]{\fnms{Alejandro F.} \snm{Ram\'irez}\ead[label=e2,mark]{ar23@nyu.edu}}
\author[C]{\fnms{Santiago} \snm{Saglietti}\ead[label=e3,mark]{saglietti.sj@uc.cl}}
\and
\author[A]{\fnms{Sebasti\'an} \snm{Zaninovich}\ead[label=e4,mark]{szaninovich@dm.uba.ar}}
%%%%%%%%%%%%%%%%%%%%%%%%%%%%%%%%%%%%%%%%%%%%%%
%% Addresses                                %%
%%%%%%%%%%%%%%%%%%%%%%%%%%%%%%%%%%%%%%%%%%%%%%
\address[A]{Fac. Cs. Exactas y Naturales, Universidad de Buenos Aires and IMAS UBA-CONICET,
\printead{e1,e4}}

\address[B]{NYU-ECNU Institute of Mathematical Sciences at NYU Shanghai
\printead{e2}}

\address[C]{Facultad de Matemáticas, Pontificia Universidad Católica de Chile
\printead{e3}}

\end{aug}

\begin{abstract} We establish a general version of the strong KPZ universality conjecture near the axis for random walks in a random environment (RWRE) on $\mathbb Z^2$. For an i.i.d. elliptic random environment, we consider the quenched large deviations probabilities for trajectories starting at the origin and arriving at time $n+[n^a]$ to the position $(n,[n^a])$
and show that, if the logarithm of the right-jump probability has a finite moment of order $p>2$, then for $a < \frac{3}{7}(1-\frac{2}{p})$ the fluctuations of these  propabilities are asymptotically governed by the GUE Tracy–Widom distribution. Our results are based on a comparison between RWRE and a last passage percolation model, whose asymptotic fluctuations near the axis were previously established by Bodineau-Martin \cite{BM05} and Baik-Suidan \cite{BS05}.
Furthermore, we obtain also the full convergence to the directed landscape in this regime based on the extension of the aforementioned results to this setting by McKeown and Zhang \cite{MZ25}.

\end{abstract}

\begin{keyword}[class=MSC]
\kwd{60K35}
\kwd{82C41}
%\kwd[; secondary ]{}
\end{keyword}

\begin{keyword}
\kwd{random walk in random environment}
\kwd{large deviations}
\kwd{KPZ universality}
\kwd{Tracy-Widom distribution}
\kwd{fluctuations}
\end{keyword}

\end{frontmatter}

%\tableofcontents

\section{Introduction and main results}

Random walk in random environment (RWRE) is a fundamental mathematical model which describes the movement of a particle in a highly disordered
media.  In its canonical setting, the model can be defined in its quenched version as a discrete time Markov chain with state space $\mathbb Z^d$, $d\ge 1$,
where the particle can only jump in one time step, from a given site $x\in\mathbb Z^d$ to a nearest neighboring site $x+e\in\mathbb Z^d$, $|e|_1=1$, according to
jump probabilities that form a set of independent vectors indexed by $x\in\mathbb Z^d$, the so-called environment.  The RWRE model is then obtained by choosing the environment at random according to some prescribed distribution. When $d = 2$ and the jump probabilities are uniformly elliptic, it is conjectured that this model satisfies the strong KPZ universality conjecture: under a specific scaling, the fluctuations
of the (quenched) large deviation probabilities of the walk are asymptotically distributed as the KPZ fixed point. In this article, we prove that
a version of this conjecture near the axis is verified under the sole assumption that the random environment is i.i.d., elliptic and that the logarithm of the jump probabilities has moments of order $p>2$. To our knowledge, this is the first  KPZ universality result for RWRE {for a general class of laws of the environment},  complementing the previous findings of Barraquand and Corwin \cite{BC17}, who
established the strong KPZ universality for the integrable Beta random walk model. %\textcolor{red}{Similar results have been obtained recently in the context of (one-sided) ballistic deposition in \cite{GRSZ}}.

To state our results, we first rigorously define the RWRE model. To this end, we fix some integer $d \geq 1$ and define $U \coloneqq\{e\in\mathbb Z^d: |e|_1=1\}$ together with
\[
\mathcal P \coloneqq \left\{\lp p(e) \rp_{e\in U} \;: \; p(e) \geq 0 \ \forall\, e \in U, \sum_{e\in U}p(e)=1\right\}.
\]
We call $\Omega:=\mathcal P^{\mathbb Z^d}$ the {\it environment space} of the RWRE and each $\omega \in \Omega$ an \textit{environment}. Notice that each $\omega \in \Omega$ is of the form $\omega=(\omega(x))_{x\in\mathbb Z^d}\in \Omega$, where $\omega(x)=(\omega(x,e))_{e\in U}\in\mathcal P$. A {\it random walk
  in the environment $\omega$} starting from $x\in\mathbb Z^d$, is a Markov chain $(X_n)_{n \in \N_0}$ with transition probabilities and initial distribution respectively given by
$$
P_{x,\omega}(X_{n+1}=y+e|X_n=y)=\omega(y,e) \qquad {\rm and} \qquad P_{x,\omega}(X_0=x)=1,
$$
for all $n\ge 0$ and $y\in\mathbb Z^d$. We call $P_{x,\omega}$ the \textit{quenched law} of the RWRE. If now we choose the environment $\omega$ at random according some distribution $\P$ on $\Omega$, we obtain a measure $P_x$ on $\Omega \times (\Z^d)^{\N_0}$ given by
\[
P_x(A\times B) = \int_A P_{x,\omega}(B) \mathrm{d}\P(\omega)
\] for each pair of Borel sets $A \subseteq \Omega$ and $B \in (\Z^d)^{\N_0}$. We  call $P_x$ the \textit{annealed law} of the RWRE and, in general, we say that the sequence $(X_n)_{n \in \N_0}$ under $P_x$ is a RWRE with \textit{environment law} $\P$. In the sequel, we will say that the environment $\omega$ is:
\begin{enumerate}
    \item \textit{i.i.d.} whenever the random vectors $(\omega(x))_{x \in \Z^d}$ are i.i.d. under $\P$;
    \item \textit{elliptic} if $\P$-a.s. $\omega(x,e)>0$ for all $x,e$;
    \item \textit{uniformly elliptic} if there exists $\kappa > 0$ such that $\P$-a.s. $\omega(x,e)\ge\kappa$ for all $x,e$.
\end{enumerate}
%Let now $\mathbb P$ be a probability measure defined on the environment space $\Omega$ under
%which the random vector $(\omega(x))_{x\in\mathbb Z^d}$ is i.i.d.. {Furthermore an environment is said to be {\it elliptic} if $\mathbb P$-a.s. $\omega(x,e)>0$ for all $x,e$ while it is said to be {\it uniformly elliptic}
%if} there is a constant $\kappa\in (0,1)$ such that $\mathbb P$-a.s. $\omega(x,e)\ge\kappa$ for all $x,e$.
%We now define $P_{x,\omega}$ the {\it quenched law} of the above defined Markov chain.

In 2003, Varadhan \cite{V03} established a quenched large deviation principle for RWRE. 
More precisely, he showed that, if $(X_n)_{n \in \N_0}$ is a random walk in a uniformly elliptic i.i.d.\ random environment, then 
there exists a lower semicontinuous function \(I_q:\mathbb{R}^d\to[0,\infty]\) such that, for every Borel set \(G\subset\mathbb{R}^d\),
\[
-\inf_{x\in G^\circ}I_q(x)
\leq \liminf_{n\to\infty}\frac{1}{n}\log P_{0,\omega}\left(\frac{X_n}{n}\in G\right)
\le \limsup_{n\to\infty}\frac{1}{n}\log P_{0,\omega}\left(\frac{X_n}{n}\in G\right)\le -\inf_{x\in \bar G}I_q(x),
\]
for $\mathbb P$-a.s. every $\omega\in\Omega$. Furthermore, if $\mathbb D:=\{x\in\mathbb R^d: |x|_1 \leq 1\}$, then $I_q(x)<\infty$
if and only if $x\in\mathbb D$ and $I_q$ is continuous in $\mathbb D^\circ$. This result was strengthened by Yilmaz in \cite{Y09}, weakening
the assumption  of uniform ellipticity to the assumption that {the environment is elliptic and that} there exists some {$p>d$} such that
\begin{equation}
  \label{five}
\mathbb E\left[|\log \omega(0,e)|^{{p}}\right]<\infty \qquad \text{ for all  }e \in U.
\end{equation}
On the other hand, the following subgaussian fluctuation behavior is conjectured in the case $d=2$, at least in the case of uniformly elliptic environments, corresponding to point-to-point results for the models of last passage percolation (LPP) and directed polymers. {In what follows, for $x\in\mathbb R^d$ we adopt the notation
$$
[x]:=([x_1],\ldots,[x_d]).
$$}

\noindent {\bf KPZ conjecture: } Let $d=2$. Consider a random walk in an i.i.d. uniformly elliptic random environment. Let $\mathbb O:=\{x\in\mathbb D: I_q(x)=0\}$ be the zero set of the quenched large deviation rate function. Then, there exists $\epsilon>0$ such that, for all $x\in\mathbb D \setminus \lp\mathbb O\cup \{0\} \rp$,
\[
\frac{\log P_{0,\omega}(X_n=[xn])-nI_q(x)}{n^{1/3}} \overset{d}{\longrightarrow}{\rm TW}_{\rm GUE},
\]
in the limit as $n \rightarrow \infty$, where ${\rm TW}_{\rm GUE}$ denotes the Tracy--Widom distribution associated with the Gaussian Unitary Ensemble (GUE), and the convergence in distribution is with respect to the environment law $\P$.

The GUE Tracy--Widom distribution is a central object in random matrix theory, as it governs the asymmptotic fluctuations of the maximal eigenvalue of a random matrix in the Gaussian Unitary Ensemble. See \cite{TW} for further details.

In this article, we prove a version of the KPZ conjecture {\it near the axis}. Let
\[
\R^4_\uparrow \coloneqq \{(x_1,x_2,y_1,y_2) \in \R^4 \colon x_2 < y_2 \},
\]
and, for $(x,y)= ((x_1,x_2),(y_1,y_2)) \in \R^4_\uparrow$, define
\begin{equation*}
    D(x,y) := -\log \lp P_{[ x],w}(X_{|[ x ]- [ y ]|_1} = [ y ] )\rp ,
\end{equation*}
together with its normalized version
\begin{equation*}
    \widehat{D} (x,y) := \frac{D (x,y) - \mu |[x]-[y]|_1}{\sigma},
\end{equation*}
where $\mu := \E \left[ \log (\omega (0,e_1)\right]$ and $\sigma := \textrm{Var} \lp \log \lp \omega (0,e_1)\rp\rp$, with $e_1:=(1,0)$; and its rescaled version
\begin{equation}
\label{eq: L_n,alpha}
    D_{ n, a}(x,y) := n^{\frac{a - 3}{6}} \widehat{D}((x)_n,(y)_n) -2(y_2-x_2) n^{\frac{2a}{3}} - (y_1 - x_1)n^{\frac{a}{3}},
\end{equation}
where, for $z =(z_1,z_2)\in \R^2$, we write $(z)_n := \lp z_2 n + 2z_1 n^{1- \frac{a}{3}}, [ z_2n^a] \rp$. Observe that all these quantities are random variables depending on the environment $\omega$, but in all cases we omit the dependence on $\omega$ for simplicity. 

Now, let $\mathcal L$ be the \textit{directed landscape}.
%(see \cite{DOV} for a precise definition).
 Essentially, $\mathcal L$ is a random continuous~function from $\R_\uparrow^4$ to $\R$ arising as the universal scaling limit of many models of random directed metrics, such as last passage percolation with geometric or exponential weights (defined in Subsection~\ref{subsec: standard LPP} below). %The main property of $\mathcal L$ that we will use in this article is that it satisfies Theorem~\ref{theo:mckeown-zhang}. 
For a precise definition, further background and properties of the directed landscape, see \cite{DOV}. The goal of this article is to show that $D_{n,\alpha}$ converges in distribution as $n \to \infty$ to $\mathcal L$. There are various forms of convergence in law to $\cL$, all of which are discussed in detail in \cite{DV22}. The two more relevant for our purposes are those of \textit{compact} and \textit{hypograph} convergence. Compact convergence simply amounts to distributional convergence with respect to uniform convergence over compact sets. As for hypograph convergence, for our purposes it will only be important to know that it is weaker than compact convergence.  

Our main result is then the following.
\begin{theorem}
\label{theo: convergence to directed landscape}
Let us consider a random walk in a random environment $\omega$ on $\Z^2$ satisfying the following assumptions:
\begin{enumerate}
    \item [A1.] $\omega$ is i.i.d. and elliptic;
    \item [A2.] $\E[|\log \omega(0,e_1)|^p] < \infty$ for some $p > 2$.
\end{enumerate} Then, for any $0 < a < \frac{3}{7} \lp 1 - \frac{2}{p} \rp$, we have that 
\[
D_{n, a} \overset{d}{\longrightarrow} \cL
\] in the hypograph sense. If $p > 5$, then the convergence holds also in the compact sense.
\end{theorem}

As a consequence of Theorem~\ref{theo: convergence to directed landscape}, we have the following version of the KPZ conjecture near the axis. 
\begin{theorem}
  \label{theorem1} Let us consider a random walk in a random environment $\omega$ on $\Z^2$ under the same hypotheses of Theorem \ref{theo: convergence to directed landscape} and define 
\[
\sigma :={\rm Var}(\log\omega(0,e_1)).
\]
Then, the following limits hold:
  \begin{itemize}
  \item[(i)] If $a < \frac{6}{7}(\frac{1}{2}-\frac{1}{p})$, then, as $n\to\infty$,
\[
\frac{\log P_{0,\omega}(X_{n+[n^a]}=(n,[n^a]))-n I_q(e_1) -2\sigma\sqrt{n^{1+a}}}{\sigma\sqrt{n^{1-\frac{a}{3}}}} \overset{d}{\longrightarrow}  {\rm TW}_{\rm GUE}.
\]
\item[(ii)] For any fixed $k \in \N$, as $n\to\infty$,
\[
  \frac{\log P_{0,\omega}(X_{n+k}=(n,k))-n I_q(e_1) }{\sqrt{n}} \overset{d}{\longrightarrow} \lambda_k,
\]
where $\lambda_k$ is the largest eigenvalue of a $k\times k$ \rm{GUE} random matrix. 
\item[(iii)] If $a < \frac{1}{2}-\frac{1}{p}$, then, as $n \to \infty$,
\[
\frac{\log P_{0,\omega}(X_{n+[n^a]}=(n,[n^a]))-nI_q(e_1)-2\sigma\sqrt{n^{1+a}}}{\sqrt{n}}\overset{d}{\longrightarrow} 0.
\] 
\end{itemize}
\end{theorem}

The assertion in (i) above confirms the validity of the KPZ conjecture for points near (but not too near) the $x$ axis. On other hand, (ii) shows that fluctuations \textit{very} near the axis do not satisfy the KPZ conjecture, and in fact scale as in the standard central limit theorem (CLT), albeit with non-Gaussian asymptotic distribution (unless $k=1$). Finally, (iii) shows that, for points further away from the axis, fluctuations are still of smaller size than that in the CLT. We remark that an analogous result holds also for the $y$-axis, with the appropriate changes in the moment hypotheses. The proof is analogous.

{There is a special case in which the environment is given by Beta random variables, the so-called \textit{Beta random walk}. In this case, the walk is allowed to jump only to the right with a probability given by a Beta random variable $B(\alpha,\beta)$ of parameters $\alpha>0$ and $\beta>0$ and upwards with a probability given by $1-B(\alpha,\beta)$. For this environment, the strong KPZ conjecture was proven by Barraquand and Corwin \cite{BC17} for the case $\alpha=\beta=1$, then extended  by Koroktchik \cite{K22} to the case $\alpha=1$ and $\beta>0$ and by Oviedo, Panizo and Ram\'\i rez \cite{OPR22} for the case $\alpha>0.7$ and $\beta>0$.}
On the other hand, since for $\alpha>0$ and $\beta>0$, $\log B(\alpha,\beta)$ has moments of arbitrary order, Theorem \ref{theorem1} includes the case of a Beta random walk for arbitrary values of $\alpha>0$ and $\beta>0$, complementing the previous results \cite{BC17,K22,OPR22}.
%Furthermore, an analogous result with a similar proof should hold for the model of directed polymers on $\mathbb Z^2$, but for the sake of clarity and simplicity of presentation, we have decided here to present only the RWRE case.

% \textcolor{red}{It should be noted that for a particular law of the random environment, where the walk is allowed to jump only to the right with a probability given by  a beta random variable $B(\alpha,\beta)$ of parameters $\alpha>0$ and $\beta>0$ and upwards with a probability given by $1-B(\alpha,\beta)$, the so called beta-random walk, the strong KPZ universality was proven by Barraquand and Corwin \cite{BC17} for the case $\alpha=\beta=1$, then extended  by Koroktchik \cite{K22} to the case $\alpha=1$ and $\beta>0$ and by Oviedo, Panizo and Ram\'\i rez \cite{OPR22} for the case $\alpha>0.7$ and $\beta>0$.}

For directed models at zero temperature, this type of fluctuation limits near the axis are already known: Bodineau and Martin \cite{BM05} and Baik and Suidan \cite{BS05} proved an analogue of Theorem~\ref{theorem1} for LPP with general weights under similar moment conditions and McKeown and Zhang \cite{MZ25} subsequently generalized these results to obtain the analogue of Theorem~\ref{theo: convergence to directed landscape}, convergence to the directed landscape. Analogous results for a directed variation of ballistic deposition were proved by the present authors in \cite{GRSZ}.

For directed polymers, Auffinger, Baik and Corwin \cite{ABC12} proved convergence to the $\mathrm{TW}_{\mathrm{GUE}}$ distribution for the fluctuations of the partition function in complete analogy to our Theorem~\ref{theorem1}, but under a finite fourth moment assumption. Their proofs differ with ours in the couplings used to bound the fluctuations (and also in that they do not show convergence to the directed landscape). The approach developed here can also be applied to the directed polymer model, yielding a slight improvement in the range of~$a$. Whereas \cite{ABC12} requires $a<3/14$, our techniques work up to $a<\frac3 7(1-\frac2p)$, under the sole assumption $p>2$. Moreover, with our method we can also obtain convergence to the directed landscape in this setting as in Theorem~\ref{theo: convergence to directed landscape}.

The proof of Theorem \ref{theo: convergence to directed landscape}, given in the next section, is based on a comparison between RWRE and the LPP model.
Bodineau and Martin \cite{BM05} and Baik and Suidan \cite{BS05} proved independently a universality result for general LPP models near the axis, showing that the asymptotic fluctuations of the marginals are given by the GUE Tracy-Widom distribution. This was recently extended by McKeown and Zhang \cite{MZ25} to full convergence to the directed landscape $\mathcal L$ in the same regime. Using these results in combination with our comparison results, i.e. Lemma \ref{lemma: L - L tilde} below, immediately yields Theorem \ref{theo: convergence to directed landscape}. The proof of Theorem \ref{theorem1}, which is similar, is given in Section \ref{proof.thm1}.

  \section{Proof of Theorem \ref{theo: convergence to directed landscape}}
  \label{sec: Proof of main theorem} %ex s2
This section is devoted to the proof of Theorem~\ref{theo: convergence to directed landscape}. As discussed in the Introduction, we obtain the result by comparing our process with a specific LPP model for which similar results have already been established. We do this comparison in two steps, via an auxiliary process we call the LPP model corresponding to an environment. In the following subsections we present both LPP models to be used in this two-step comparison.
  
  \subsection{Last Passage Percolation model corresponding to an environment}
  \label{subsec: LPP associated to RWRE}%ex s2.1
Consider the directed graph $(\mathbb Z^2, \mathbb E^2)$, where $\mathbb E^2:=
  \{ (x,y) : x,y\in\mathbb Z^2, |x-y|_1=1\}$. Given an environment $\omega$, we define the {\it weights} of the edges of the graph
  $(\mathbb Z^2,\mathbb E^2)$ by
\[
W_{(x,x+e)}:=-\log \omega(x,e), \qquad {\rm for}\ {\rm all}\ (x,x+e)\in\mathbb E^2.
\]
%Given any two vertices $x,y\in\mathbb Z^2$, we will say that $y$ is nearest upward-rightward  neighbor of $x$ if $y-x \in \ls (1,0),(0,1)\rs$. 
Define an \textit{up-right path} from $x$ to $y$ to be a vector $(x_0,...,x_{n}) \in (\Z^2)^{n+1}$ for some $n \in \N$ such that $x_0 = x$, $x_n = y$ and $x_{i} - x_{i-1} \in \ls (1,0),(0,1)\rs$ for all $i \in \ls 1,...,n\rs$. The set of up-right paths from $x$ to $y$ is given by
\begin{equation}
\label{eq: uprights paths set}
\cG(x,y):=\ls (x_0,\ldots,x_n): x_0=x, \ x_n=y, \ (x_0,...,x_n)\text{ is an up-right path}\rs .
\end{equation}
Now, given a path $\gamma=(x_0,\ldots,x_n)\in\mathcal G(x,y)$, define the {\it passage time of the path $\gamma$} as
\[
t_\gamma := \sum_{i=0}^{n-1} W_{(x_i,x_{i+1})}.
\]
Also, define the {\it last passage time} between $x$ and $y$ as
\[
T(x,y) := \sup \{ t_\gamma : \gamma \in \mathcal G(x,y) \}.
\]
Finally, we say that a path $\gamma^* \in \mathcal G(x,y)$ is a {\it geodesic path} between $x$ and $y$ if
\[
T(x,y) = t_{\gamma^*}.
\]

Since $\mathcal G(x,y)$ is finite, there always exists at least one geodesic path between $x$ and $y$. For certain choices of the weights, the geodesic may not be unique.
We call the above model the {\em LPP model corresponding to the environment $\omega$}.
The following lemma is the key step in the proof of Theorem \ref{theorem1}. 

\begin{lemma}
  \label{lemma: comparision rwre with associated lpp}
  Consider  a random walk in the environment $\omega$ and its corresponding LPP model.
 Then, given $n \in \N$, for every $x,y\in\mathbb Z^2$ such that $|x-y|_1=n$, we have that
  \begin{equation}
    \label{one}
\mathrm{e}^{-T(x,y)}\le P_{x,\omega}(X_n=y)\le n^{|x_2-y_2|} \mathrm{e}^{-T(x,y)}.
  \end{equation}
\end{lemma}

\begin{proof} Note that, if $\gamma^*=(x_0^*,\ldots,x_n^*)$ is any geodesic path between $x$ and $y$, then
  $$
  P_{x,\omega}(X_n=y)=\sum_{\gamma\in\mathcal G(x,y)} \prod_{i=0}^n \omega(x_i,x_{i+1}-x_i)
  \ge \prod_{i=0}^n \omega(x^*_i,x^*_{i+1}-x^*_i)=e^{-t_{\gamma^*}}=e^{-T(x,y)}.
  $$ This proves the lower bound of
  \eqref{one}. On the other hand,
  since, for each $(x_0,...x_n) \in~\cG(x,y)$,
\[
\mathrm{e}^{-T(x,y)}\ge \prod_{i=0}^n \omega(x_i,x_{i+1}-x_i),
\]
we have the upper bound
\begin{equation}
  \label{two}
P_{x,\omega}(X_n=y)\le |\mathcal G(x,y)| \mathrm{e}^{-T(x,y)},
\end{equation}
where $|\mathcal G(x,y)|$ is the number of up-right paths between $x$ and $y$. It is straightforward to check that
\begin{equation}
  \label{three}
|\mathcal G(x,y)| = \binom{n}{|x_2-y_2|}\le n^{|x_2-y_2|}.
\end{equation}
Hence, by substituting \eqref{three} into \eqref{two}, we conclude the proof of the upper bound.
\end{proof}

\subsection{Standard Last Passage Percolation}
\label{subsec: standard LPP}
Given a family of i.i.d. random variables $(\tau_x)_{x \in \Z^d}$, for $x,y \in \Z^2$ we define the $\tau$-\textit{last passage time} between $x$ and $y$ as
\begin{equation}
\label{eq: def tilde L}
    L(x,y) := \sup \sum_{i = 0}^n \tau_{x_i},
\end{equation}
where the supremum is taken over all $\gamma = (x_0,...,x_n) \in \cG(x,y)$. Let us define $\widehat{L}(x,y)$ and $L_{n, a}(x,y)$ in the same manner as $\widehat{D}(x,y)$ and $D_{n,a} (x,y)$ in \eqref{eq: L_n,alpha}. The main difference between an LPP model corresponding to an environment and a standard LPP model is that the random variables of the first one are assigned to the edges of $\Z^2$ while the latter assigns weights to the vertices. Also, notice that the weights assigned to the edges in the LPP model through the environment $\omega$ are neither independent, nor identically distributed. 
To prove Theorem \ref{theo: convergence to directed landscape}, we will exploit the fact that $L_{n,a}$ converges to the directed landscape for any $0< a < \frac{3}{7}(1 - \frac{2}{p})$ under appropriate moment conditions.

\begin{theorem}[\cite{MZ25}]
\label{theo:mckeown-zhang}
Let $L$ be a standard LPP model with i.i.d. weights $(\tau_x)_{x \in \Z^d}$ satisfying that $\E | \tau_{0}|^p < \infty$ for some $p > 2$. Then, for any $0 < a < \frac{3}{7} \lp 1 - \frac{2}{p} \rp$,  $L_{n, a} \xrightarrow{d} \cL$ in the hypograph sense. If $p > 5$, the convergence is also in the compact sense.
\end{theorem}

\subsection{Proof of Theorem~\ref{theo: convergence to directed landscape}}

In this section we give the proof of Theorem~\ref{theo: convergence to directed landscape}. The first element of the proof is to compare the RWRE with a suitable standard LPP model, via the LPP corresponding to the environment of the RWRE. 

\begin{lemma}
\label{lemma: L - L tilde}% Let $t > 0$ and, given $n,k \in \N$, define 
    Let $\omega$ be a random environment $\omega$ on $\Z^2$ satisfying the assumptions of Theorem~\ref{theo: convergence to directed landscape} and consider the standard LPP model of weights $\tau_{x}:=W_{(x,x+e_1)}$.
    Then, there exists a constant $C>0$, which depends only on the distribution of $\omega$, such that, for each $t > 0$ and $n,k \in \N$, we have
    \begin{equation}
    \label{eq: mean between L and tilde L}
        \E \left[ \max_{\substack{(x,y) \in \Lambda_t(n,k)}} |D(x,y) - L(x,y)| \right] \leq Ct^2k(n+k)^{\frac{1}{p}},
    \end{equation} where 
    \[
   \Lambda_t(n,k) := \left\{ (x,y) \in \Z^4 \cap \lp[0,tn]  \times [0, tk] \rp^2 \colon  x_i\leq y_i \text{ for }i \in \{1,2\} \right\}. 
    \]
    % \patu{Ojo que cambié $n$ por $n+k$ en la cota. Habría que rastrear que todos los cambios que hice en lo que sigue están ok con eso y son exhaustivos. }
%    In particular, if $a <\frac{3}{7}(1 - \frac{2}{p})$, then $L_{n,a} - \tilde{L}_{n, a} \rightarrow 0$ in the compact distributional sense.
\end{lemma}

\begin{proof}[Proof of Lemma \ref{lemma: L - L tilde}]
    By Lemma \ref{lemma: comparision rwre with associated lpp} we have that, for $(x,y)\in \Lambda_t(n,k)$,
    $$|D(x,y) - T(x,y)| \leq tk \log(tn+tk) \leq t^2k(n+k)^{\frac{1}{p}}, \quad \P\text{-a.s.}$$ Then, it is enough to prove that
    \begin{equation*}
        \E \left[ \max_{\substack{(x,y) \in \Lambda_t(t,k)}} |L(x,y) - T(x,y)| \right] \leq C t^2k(n+k)^{\frac{1}{p}}.
    \end{equation*}
    To this end, notice that, given a path $\gamma = (z_0,...,z_k) \in \cG(x,y)$, we have
    \begin{equation*}
        \left|\sum_{i} \tau_{z_i} -  \sum_{i} W_{(z_i, z_{i+1})}\right| \leq  \sum_{\substack{i \colon \\ z_{i+1}-z_i = (0,1)}} |\tau_{z_i}| + |W_{(z_i, z_{i+1})}|.
    \end{equation*}
    Thus, we can bound this quantity by
    \begin{equation}
    \label{eq: sum M_k}
        \sum_{j = 1}^{[ tk ]} M_j^{(n)},
    \end{equation}
    where $M_j^{(n)} := \max \ls \tau_x + W_{(x,x + (0,1)) } \colon 1 \leq x_1 \leq n, x_2 = j \rs$. For simplicity, call 
    \[
    Y_{i,j} := |\tau_{(i,j)}| + \left| W_{(i,j), (i,j+1)} \right|.
    \] Since \eqref{eq: sum M_k} is uniform over all paths starting at $x$ and ending at $y$  with $(x,y)\in \Lambda_t(n,k)$, we have that
    \begin{equation*}
        \max_{\substack{\substack{(x,y) \in \Lambda_t(t,k)}}} |L(x,y) - T(x,y)| \leq \sum_{j = 1}^{[ tk ]} M_j^{(n)}.
    \end{equation*}
The variable $M_k^{(n)}$ is the maximum of $n$ random variables with finite $p$-$th$ moment and thus, by Jensen's and the triangle inequalities, we have that
\begin{equation*}
    \E\left[M_k^{(n)}\right] \leq \| (M_k^{(n)})\|_p \leq n^{1/p} \|Y_{0,0}\|_p, 
\end{equation*}
and, therefore,
\begin{equation*}
    \E \left[\sum_{j=1}^{[ tk ]} M_j^{(n)}\right] \leq [tk]n^{1/p} \|Y_{0,0}\|_p,
\end{equation*}
from which the result follows.
\end{proof}

% \subsection{Standard Last Passage Percolation:} \chino{REESCRIBIR esta seccion}
% Given $x,y \in \Z^d$, a family of i.i.d. random variables $(w_{x})_{x \in \Z^d}$, a sequence $ \gamma = (x_0,... x_n) \in \lp \Z^d \rp^n$ such that $x_0 = x$, $x_n = y$ and $|x_i - x_{i-1}|_1 = 1$ for all $j = 1,...,n$, then we define the $passage$ $time$ $\tilde{t}_\gamma$ of the path $\gamma =(x_0,...,x_n)$ as
% \begin{equation*}
%     \tilde{\gamma} = \sum_{i = 0}^{n}w_{x_i}.
% \end{equation*}
% Then we define
% \begin{equation*}
%     \tilde{L}(x,y) = \sup \tilde{t}_\gamma,
% \end{equation*}
% where the supremum is taken over all $\gamma = (x_0,...x_n)$ such that $x_{i+1} - x_i \in U$ and $x_i \leq x_{i+1}$ for all $i = 0,...,n-1$. Let us define $\tilde{L}_{n,a}(x,y)$ in the same manner as $L_{n,a}$

%The main difference between the associated LPP model and a standard LPP model is that the law of the vertical and horizontal weights is not necessarily the same in the RWRE version.

We are now ready to prove Theorem \ref{theo: convergence to directed landscape}.

\begin{proof}[Proof of Theorem \ref{theo: convergence to directed landscape}]
Let $t ,\delta > 0$. By Lemma \ref{lemma: L - L tilde}, we get that
\begin{align*}
    &\P\Bigg(
        \max_{(x,y)\in\big[-\tfrac{t}{2},\tfrac{t}{2}\big]^4}
        \big| D_{n,a}(x,y)- L_{n,a}(x,y) \big|
        \ge \delta
    \Bigg)
    \\
    \le\;&
    \P\Bigg(
        \max_{\substack{(x,y)\in \Lambda_t(n,n^a)}}
        \big|D(x,y)- L(x,y)\big|
        \ge \sigma \delta\, n^{\frac12-\frac{a}{6}}
    \Bigg)
    \\
    \le\;&
        \E \left[ \max_{\substack{(x,y) \in \Lambda_t(n,n^a)}} |D(x,y) - L(x,y)| \right]  \sigma \delta^{-1} n^{-(a+\frac1p+\varepsilon)}\\
      \le  \; & C \sigma t^2n^a(n+n^a)^{\frac{1}{p}} \delta^{-1} n^{-(a+\frac1p+\varepsilon)}\\
            \le  \; & C \sigma t^2 \delta^{-1} n^{-\varepsilon}.
\end{align*} Above, in the third line, we chose $\varepsilon>0$ small enough so that $a + \frac{1}{p } +\varepsilon< \frac{1}{2} - \frac{a}{6}$ and then used Markov's inequality while, for the bound in the last line, we used Lemma~\ref{lemma: L - L tilde}.  This implies that $D_{n,a} -~L_{n,a}$ converges in probability with respect to the topology of uniform convergence over compact sets. In combination with Theorem \ref{theo:mckeown-zhang}, this implies Theorem~\ref{theo: convergence to directed landscape}.
\end{proof}

\section{Proof of Theorem \ref{theorem1}}
\label{proof.thm1}
For the proof of Theorem \ref{theorem1}, we will use Lemma~\ref{lemma: L - L tilde} together with an analogue of Theorem \ref{theorem1} but for the standard LPP model. The latter is essentially contained in the works of Bodineau and Martin \cite{BM05} and Baik and Suidan \cite{BS05}. Indeed, if for each $n \in \N$ we let $\lambda_n$ denote the largest eigenvalue of a $n \times n$ \rm{GUE} matrix, then it is well known (see e.g. \cite{TW}) that 
\begin{equation}\label{eq:bm1}
n^{\frac{1}{6}}(\lambda_n -2\sqrt{n}) \overset{d}{\longrightarrow} \TW_{\GUE}.
\end{equation} Moreover, in \cite{BM05} it is shown that, for a standard LPP model with i.i.d. weights $(\tau_x)_{x \in \Z^2}$ satisfying $\E[|\tau_0|^p]< \infty$ for some $p>2$, then, for any sequence $(k_n)_{n \in \N} \subseteq \N$ there exists a coupling between $L((0,0),(n,k_n))$ and  $\lambda_{k_n}$ such that, as $n \to \infty$,
\begin{equation}\label{eq:bm2}
\frac{L((0,0),(n,k_n)) - \mu(n+k_n) - \sigma\sqrt{n}\lambda_{k_n}}{\sigma k_nn^{1/p+\varepsilon}} \overset{d}{\longrightarrow} 0
\end{equation} for each $\varepsilon > 0$, where above $\mu := \mathbb E[\tau_0]$ and $\sigma^2:= {\rm Var} (\tau_{0})$. Upon combining \eqref{eq:bm1}--\eqref{eq:bm2}, we arrive at the following result.

\begin{theorem}
  \label{theo: Bodineau and Martin}
  Consider a standard LPP model as in Subsection  \ref{subsec: standard LPP}. Suppose that $\mathbb E[| \tau_{0}|^p]< \infty$ for some $p > 2$ and define $\mu := \mathbb E[\tau_0]$ and $\sigma^2:= {\rm Var} (\tau_{0})$.
Then, the following limits hold:
  \begin{itemize}
  \item[(i)] If $a < \frac{6}{7}(\frac{1}{2}-\frac{1}{p})$, then, as $n\to\infty$,
\[
\frac{L((0,0),(n, [n^a]))-n\mu-2\sigma \sqrt{n^{1+a}}}{\sigma \sqrt{n^{1-\frac{a}{3}}}} \overset{d}{\longrightarrow} {\rm TW}_{\rm GUE}.
\]
%In particular, if the $\log \tau_0$  has finite moments of all orders, the second condition on $k_n$ can be weakened to 
%     $k_n= o\left(n^{\frac{3}{7}}\right)$.
    \item[(ii)] For any fixed $k \in \N$, as $n\to\infty$,
\[
\frac{L((0,0),(n, k))-n\mu}{ \sigma\sqrt{n}}\to \lambda_k,
\]
where $\lambda_k$ is the largest eigenvalue of a $k\times k$ \rm{GUE} random matrix.
\item [iii)] If $a < \frac{1}{2}-\frac{1}{p}$, then, as $n \to \infty$, 
\[
\frac{L((0,0),(n,[n^a]))-n\mu-2\sigma\sqrt{n^{1+a}}}{\sqrt{n}}\overset{d}{\longrightarrow} 0.
\] 
\end{itemize}
\end{theorem}

With Theorem~\ref{theo: Bodineau and Martin} at our disposal, we are ready to prove Theorem~\ref{theorem1}.

\begin{proof}[Proof of Theorem \ref{theorem1}]
    Item (i) is a direct corollary of Theorem \ref{theo: convergence to directed landscape} in the case $x = (0,0)$ and $y = (0,1)$ (alternatively, we can use Lemma~\ref{lemma: L - L tilde} and (i) in Theorem \ref{theo: Bodineau and Martin}). Items (ii) and (iii) are a consequence of Lemma \ref{lemma: L - L tilde} combined with Theorem \ref{theo: Bodineau and Martin} items (ii) and (iii), respectively.
\end{proof}

\noindent {\bf Acknowledgements.}
Part of this work was carried out during visits of some of the authors to NYU Shanghai, Pontificia Universidad Cat\'olica de Chile and Universidad de Buenos Aires. The authors would like to thank the institutions for their hospitality and financial support.
Alejandro Ram\'irez was partially supported by NFSC 12471147 grant and by NYU Shanghai Boost Fund.
Pablo Groisman and Sebasti\'an Zaninovich were partially supported by CONICET Grant PIP 2021 11220200102825CO, UBACyT Grant 20020190100293BA and PICT 2021-00113 from Agencia I+D. Santiago Saglietti was partially supported by Fondecyt Grant 1240848.

\bibliographystyle{abbrv}
\bibliography{biblio}
\end{document}